\documentclass{amsart}
\usepackage{amsfonts}
\usepackage{amsmath,amssymb}
\usepackage{amsthm}
\usepackage{amscd}
\usepackage{graphics}
\usepackage{graphicx}
{
\theoremstyle{definition}
\newtheorem{Def}{Definition}
\newtheorem{Ex}{Example}
\newtheorem{Rem}{Remark}

}
\newtheorem{Cor}{Corollary}
\newtheorem{Prop}{Proposition}
\newtheorem{Thm}{Theorem}

\begin{document}
\title{Explicit round fold maps on some fundamental manifolds}
\author{Naoki Kitazawa}
\address{Department of Mathematics, Tokyo Institute of Technology, 
2-12-1 Ookayama, Meguro-ku, Tokyo 152-8551, JAPAN \\
Tel +81-(0)3-5734-2205, \\
Fax +81-(0)3-5734-2738,
}
\email{kitazawa.n.aa@m.titech.ac.jp}
\begin{abstract}
{\it {\rm (}Stable{\rm )} fold maps} are fundamental tools in a generalization of
 the theory of Morse functions on smooth manifolds
 and its application to studies of geometry of smooth manifolds. Constructions of explicit fold maps are important in the theory
 of stable fold maps and have been difficult. 

\ \ \ Succeeding in constructions of explicit fold maps will help us to study geometry of manifolds by using the
 geometric theory of fold maps with good geometric properties. For this purpose, in this paper, we construct
 more new explicit fold maps on some fundamental manifolds. More precisely, we construct stable
 fold maps such that the sets of all the singular values of the maps are concentric
 spheres ({\it round} fold maps). Such maps were introduced in 2013--4 and some examples have been constructed by the author. In this paper, we
 construct new examples and obtain their source manifolds by applying methods which were used in these constructions with additional new algebraic and differential
 topological techniques. 
\end{abstract}
\maketitle

\section{Introduction and terminologies}
{\it Fold maps} are fundamental tools in a generalization of the theory of Morse functions and its application to
 studies of geometry of manifolds, which is defined as a smooth map such that each singular point is of the form
$$(x_1, \cdots, x_m) \mapsto (x_1,\cdots,x_{n-1},\sum_{k=n}^{m-i}{x_k}^2-\sum_{k=m-i+1}^{m}{x_k}^2)$$ for
 two positive integers $m \geq n$ and an integer $0 \leq i \leq \frac{m-n+1}{2}$; note that $i$ is determined
 uniquely for each singular point (we call $i$ the {\it index} of the singular point). Studies of
 such maps were started by Whitney (\cite{whitney}) and Thom (\cite{thom}) in the 1950s. A Morse function
 is regarded as a fold map in the case where $n=1$ holds and for a fold map from
 a closed smooth manifold of dimension $m$ into a smooth manifold of dimension $n$
 without boundary with $m \geq n \geq 1$, the following two hold.
\begin{enumerate}
\item The {\it singular set}, which is defined as the set of all the singular points, and the set of all the singular points of index $i$ are closed smooth submanifolds of dimension $n-1$ of the source manifold. 
\item The restriction map to the singular set is a smooth immersion of codimension $1$.
\end{enumerate}
Although Morse functions exist densely on all smooth manifolds, there exist families of (closed) smooth manifolds admitting no fold maps into the
 Euclidean space of a dimension. For example, a closed smooth manifold of dimension $k \geq 2$ admits a fold map into the plane if and only if the Euler number of the manifold is even. In
 \cite{eliashberg}, \cite{eliashberg2} and other proceeding papers, more general existence problems for fold maps were studied. As a simplest example, it has been known that
 smooth homotopy spheres of dimension $k$ (not necessarily diffeomorphic to the standard
 sphere $S^k$) admits a fold map into the ${k}^{\prime}$-dimensional Eulidean space ${\mathbb{R}}^{{k}^{\prime}}$ for any integer $1 \leq {k}^{\prime} \leq k$.
   
\ \ \ Since around the 1990s, fold maps with additional conditions have been
 actively studied. For example, in \cite{burletderham}, \cite{furuyaporto}, \cite{saeki2}, \cite{saeki3}, \cite{saekisakuma} and \cite{sakuma}, {\it special
 generic} maps, which are defined as fold maps whose singular points are of the form
$$(x_1, \cdots, x_m) \mapsto (x_1,\cdots,x_{n-1},\sum_{k=n}^{m}{x_k}^2)$$ for two positive integers $m \geq n$, were studied. In \cite{sakuma}, Sakuma studied
 {\it simple} fold maps, which are defined as fold maps such that the inverse images of singular values do not have any connected component with more
 than one singular points (see also \cite{saeki}). In
 \cite{kobayashisaeki}, Kobayashi and Saeki
 investigated topological properties of {\it stable} maps into the plane including fold maps which are stable (for {\it stable} maps, see \cite{golubitskyguillemin} for example). In \cite{saekisuzuoka}, Saeki and Suzuoka
 found good topological properties of manifolds admitting stable maps such that the inverse images of regular values are disjoint unions of spheres.

\ \ \ Later, in \cite{kitazawa2}, {\it round} fold maps, which will be mainly studied in this paper, were introduced. A {\it round} fold map
 is defined as a fold map satisfying the following three.
\begin{enumerate}
\item The singular set is a disjoint union of standard spheres.
\item The restriction map to the singular set is an embedding.
\item The image of the singular set is a disjoint union of spheres embedded concentrically. 
\end{enumerate}
 
For example, some special generic maps on spheres are round fold maps whose singular sets are connected. Any standard sphere whose dimension is $m>1$
 admits such a map into ${\mathbb{R}}^n$ with $m \geq n \geq 2$ and any smooth homotopy sphere whose dimension is larger than $1$ and not $4$ admits
 such a map into the plane. For some pair of dimensions $m \geq n \geq 2$, some $m$-dimensional smooth homotopy spheres do not admit such
 maps into ${\mathbb{R}}^n$ even though they admit fold maps into the Euclidean space as mentioned in the presentation of existence problems before. See also \cite{saeki2} and Example \ref{ex:1} of the present paper.

\ \ \ In \cite{kitazawa3}, homology groups and
 homotopy groups of manifolds admitting round fold maps are studied. In \cite{kitazawa} and
 \cite{kitazawa4}, under appropriate conditions, the homeomorphism and diffeomorphism types of
 manifolds admitting round fold maps were studied under appropriate conditions. Moreover, explicit examples of (round) fold
 maps on these manifolds have been constructed in the proofs of the results. By the way, it is a fundamental and difficult problem in the theory of fold maps to construct explicit
 fold maps on explicit manifolds, although existence problems for fold maps into Euclidean spaces on these manifolds
 have been solved in the studies of general existence problems explained before. In the present paper, we obtain
 more new examples of round fold maps with the homeomorphism and diffeomorphism types of the source manifolds by applying methods used in these
 constructions with new algebraic and differential topological techniques.

\ \ \ This
 paper is organized as follows.

\ \ \ In section \ref{sec:2}, we recall {\it round} fold maps and some terminologies
 on round fold maps such as {\it axes} and {\it proper cores}. We also recall a {\it $C^{\infty}$ trivial} round fold map, which is defined as a round
 fold map whose differential topological structure satisfies a kind of triviality.

\ \ \ In section \ref{sec:3}, we study the homeomorphism and diffeomorphism types of manifolds admitting round fold maps. Mainly, we
 give new examples of round fold maps with the diffeomorphism types of their source manifolds.

\ \ \ First, under appropriate conditions, we construct a new round fold map on a manifold represented as a connected sum
 of two closed and connected $m$-dimensional manifolds admitting round fold maps into ${\mathbb{R}}^n$ ($m \geq n \geq 2$) with $m> n$ in Proposition \ref{prop:1}, which has been
 shown in \cite{kitazawa}, \cite{kitazawa2} and \cite{kitazawa4} under the assumption that $m \geq 2n$ holds. As an application, we
 construct round fold maps into ${\mathbb{R}}^4$ on
 $7$-dimensional smooth homotopy spheres by applying this construction with known facts on $7$-dimensional homotopy spheres of \cite{eellskuiper}, \cite{kervairemilnor} and \cite{milnor} (Theorem \ref{thm:1}).
 Second, we introduce another easy application of Proposition \ref{prop:1} as Theorem \ref{thm:2}. Third, by
 applying our Proposition \ref{prop:1}, on a manifold represented as a connected
 sum of a manifold having the structure of a bundle over the standard $n$-dimensional sphere $S^n$ with $m-n \geq 1$ and $n \geq 2$ and a
 manifold admitting a round fold map satisfying appropriate differential topological conditions, we construct
 a new round fold map into ${\mathbb{R}}^n$ (Theorem \ref{thm:3}). As an application of this theorem, for example, on a manifold represented as a connected
 sum of a finite number of $m$-dimensional smooth manifolds having the structures of bundles
 over the standard $n$-dimensional sphere $S^n$ whose fibers are diffeomorphic to the ($m-n$)-dimensional standard
 sphere $S^{m-n}$ under the condition that $m-n \geq 1$ and $n \geq 2$ hold, we construct
 a round fold map into ${\mathbb{R}}^n$ (Theorem \ref{thm:5} \ref{thm:5.1}), which has been obtained also in the three
 listed papers by the author under the assumption that $m \geq 2n$ holds and by considering this result and known facts on $7$-dimensional homotopy spheres of \cite{eellskuiper}
 again, we classify $7$-dimensional homotopy spheres admitting such round fold maps whose singular sets consist of one, two and three connected components, respectively in Theorem \ref{thm:6}. Last, under
 appropriate conditions a bit different from the previous conditions, we construct a new round fold map on a manifold represented as a connected sum of
 two closed and connected manifolds admitting round fold maps
 again in Proposition \ref{prop:2}. As applications, we show Theorem \ref{thm:7}, which states that a manifold represented as a connected sum of a smooth homotopy sphere
 of dimension $7$ and a $7$-dimensional smooth manifold admitting a round fold map into ${\mathbb{R}}^4$ admits
 another round fold map into ${\mathbb{R}}^4$ and Theorems \ref{thm:8} and \ref{thm:9}, which state
 that for the pair $(m,n)=(3,2),(7,4),(15,8)$, on $m$-dimensional manifolds represented as connected sums of manifolds
 admitting round fold maps into ${\mathbb{R}}^n$ with additional
 appropriate algebraic and differential topological conditions, we can construct new round fold maps. 

\ \ \ In the last section \ref{sec:4}, under appropriate assumptions, we show that for two closed and connected $m$-dimensional manifolds
 admitting $C^{\infty}$ trivial round fold map into ${\mathbb{R}}^n$ with $m>n \geq 2$ assumed, we can construct a $C^{\infty}$ trivial
 round fold map from a manifold represented as a connected sum of the manifolds into ${\mathbb{R}}^n$ (Theorems \ref{thm:10}-\ref{thm:12}).

\ \ \ We note about terminologies on spaces and maps in this paper.

\ \ \ On a topological space $X$, we denote the {\it identity map} on $X$ by ${\rm id}_X$. If the space $X$ is a topological manifold, then we denote the {\it interior} of $X$ by ${\rm Int} X$, the {\rm closure} of $X$ by $\overline{X}$ and the {\it boundary} of $X$ by $\partial X$. For two topological spaces $X_1$ and $X_2$, we denote their {\it disjoint union} by $X_1 \sqcup X_2$. For
 a map $c:X_1 \rightarrow X_2$ and subspaces $Y_1 \subset X_1$ and $Y_2 \subset X_2$ such that $c(Y_1) \subset Y_2$ holds, $c {\mid}_{Y_1}:Y_1 \rightarrow Y_2$ is the {\it restriction map} of $c$ to $Y_1$. For
 a homeomorphism $\phi:Y_2 \rightarrow Y_1$ in the same situation, by gluing
$X_1$ and $X_2$ together by $\phi$, we obtain a new topological space and denote the space by $X_1 {\bigcup}_{\phi} X_2$. 
%
%
We often omit $\phi$ of $X_1 {\bigcup}_{\phi} X_2$ and denote it by $X_1 {\bigcup} X_2$ in case we consider a natural identification. As before, for
 a smooth map $c$, we define the {\it singular set} of $c$ by the set consisting
 of all the singular points of $c$ and denote this set by $S(c)$. In addition, as before, for the map $c$, we call
 the set $c(S(c))$ the {\it singular value set} of $c$ and a {\it regular fiber} of $c$ means the fiber of a point of a regular value of the map.

\ \ \ Throughout this paper, we
 assume that $M$ is a closed smooth manifold of dimension $m$, that $N$ is a smooth manifold of
 dimension $n$ without boundary, that $f:M \rightarrow N$ is a smooth map and that $m \geq n \geq 1$ holds. In the proceeding sections, manifolds, maps between manifolds and (closed) tubular neighborhoods of submanifolds of manifolds are of class $C^{\infty}$ and in addition, for bundles whose fibers are
 ($C^{\infty}$) manifolds, the structure groups consist of ($C^{\infty}$) diffeomorphisms on the fibers unless otherwise stated. Moreover, for a manifold $X$, an {\it $X$-bundle} means a bundle whose fiber is diffeomorphic to $X$.  

\ \ \ This paper is partially based on the doctoral dissertation of the author \cite{kitazawa2}. For example, Theorems \ref{thm:1}, \ref{thm:3}, \ref{thm:4}, \ref{thm:5} and \ref{thm:7} are also
 in the doctoral dissertation.

\section{Preliminaries on round fold maps}
\label{sec:2}

In this section, we review {\it round} fold maps. See also \cite{kitazawa2} and \cite{kitazawa3} for example. \\
\ \ \ First we recall {\it $C^{\infty}$ equivalence}. For two $C^{\infty}$ maps $f_1:X_1 \rightarrow Y_1$ and $f_2:X_2 \rightarrow Y_2$, we say
 that they are {\it $C^{\infty}$ equivalent} if there exist diffeomorphisms
 ${\phi}_X:X_1 \rightarrow X_2$ and ${\phi}_Y:Y_1 \rightarrow Y_2$ such that the following diagram commutes.

$$
\begin{CD}
X_1 @> {\phi}_X  >>X_2 \\
@VV f_1 V @VV f_2 V\\
Y_1 @> {\phi}_Y  >> Y_2
\end{CD}
$$

For $C^{\infty}$ equivalence, see also \cite{golubitskyguillemin} for example.

\begin{Def}[round fold map(\cite{kitazawa2})]
\label{def:1}
$f:M \rightarrow {\mathbb{R}}^n$ ($n \geq 2$) is said to be a {\it round} fold map if $f$ is $C^{\infty}$ equivalent to
 a fold map $f_0:M_0 \rightarrow {\mathbb{R}}^n$ on a closed manifold $M_0$ such that the following three hold.

\begin{enumerate}
\item The singular set $S(f_0)$ is a disjoint union of ($n-1$)-dimensional standard spheres and consists of $l \in \mathbb{N}$ connected components.
\item The restriction map $f_0 {\mid}_{S(f_0)}$ is an embedding.
\item Let ${D^n}_r:=\{(x_1,\cdots,x_n) \in {\mathbb{R}}^n \mid {\sum}_{k=1}^{n}{x_k}^2 \leq r \}$. Then, $f_0(S(f_0))={\sqcup}_{k=1}^{l} \partial {D^n}_k$ holds.  
\end{enumerate}

We call $f_0$ a {\it normal form} of $f$. We call a ray $L$ from $0 \in {\mathbb{R}}^n$ an {\it axis} of $f_0$ and
 ${D^n}_{\frac{1}{2}}$ the {\it proper core} of $f_0$. Suppose that for a round fold map $f$, its normal form $f_0$ and diffeomorphisms
 $\Phi:M \rightarrow M_0$ and $\phi:{\mathbb{R}}^n \rightarrow {\mathbb{R}}^n$, the relation $\phi \circ f=f_0 \circ \Phi$ holds. Then,
 for an axis $L$ of $f_0$, we also call ${\phi}^{-1}(L)$ an {\it axis} of $f$ and for the proper core ${D^n}_{\frac{1}{2}}$ of $f_0$, we
 also call ${\phi}^{-1}({D^n}_{\frac{1}{2}})$ a {\it proper core} of $f$. 
\end{Def}

For a round fold map $f:M \rightarrow {\mathbb{R}}^n$ and for any connected component $C$ of the singular value set of $f(S(f))$, there exists a small
 smooth closed tubular neighborhood $N(C)$ regarded as a product bundle $C \times [-1,1]$ over $C$ such that the composition
 of the restriction map to the set $f^{-1}(N(C))$ of $f$ and the projection onto $C$ is a submersion and gives $f^{-1}(N(C))$ the structure
 of a bundle over $C$ whose fiber is a compact manifold. Especially, if $C$ is the image of a connected component of the singular set consisting of points
 of index $0$, then the resulting bundle is a bundle whose fiber is the ($m-n+1$)-dimensional standard closed disc $D^{m-n+1}$ and whose structure
 group consists of linear transformations. 

\ \ \ In this paper, we call a bundle whose fiber is a standard disc (sphere) and whose structure group consists of linear
 transformations a {\it linear} bundle such as this bundle. For an integer $k \geq 2$, we denote the $k$-th special linear group, which is regarded as the group of
 all the linear transformations on the disc $D^{k}$ (the sphere $S^{k-1}$), by $SO(k)$. \\
 
\ \ \ Let $f$ be a normal form of a round fold map and $P_1:={D^{n}}_{\frac{1}{2}}$. We set $E:=f^{-1}(P_1)$ and $E^{\prime}:=M-f^{-1}({\rm Int} P_1)$. We set
 $F:=f^{-1}(p)$
 for $p \in \partial P_1$. We put $P_2:={\mathbb{R}}^n-{\rm Int} P_1$. Let $f_1:=f {\mid}_{E}:E \rightarrow P_1$ if $F$ is non-empty and let $f_2:=f {\mid}_{{E}^{\prime}}:{E}^{\prime} \rightarrow P_2$.

\ \ \ $f_1$ gives the structure of a trivial bundle over $P_1$ and ${f_1} {\mid}_{\partial E}:\partial E \rightarrow \partial P_1$ gives
 the structure of a trivial bundle
 over $\partial P_1$ if $F$ is non-empty. $f_2 {\mid}_{\partial {E}^{\prime}}:\partial {E}^{\prime} \rightarrow \partial P_2$
 gives the structure of a trivial bundle over $\partial P_2$.

\ \ \ We can give ${E}^{\prime}$ the structure of a bundle over $\partial P_2$ as follows. 

\ \ \ Since for ${\pi}_P(x):=\frac{1}{2} \frac{x}{|x|}$ ($x \in P_2$), ${\pi}_P \circ f {\mid}_{{E}^{\prime}}$ is
 a proper submersion, this map gives ${E}^{\prime}$ the structure of
 a $f^{-1}(L)$-bundle over $\partial P_2$ (apply Ehresmann's fibration
 theorem \cite{ehresmann}). We call this bundle the {\it surrounding bundle} of $f$. Note that the structure group of this bundle
 is regarded as the group of diffeomorphisms on $f^{-1}(L)$ preserving the function $f {\mid}_{f^{-1}(L)}:f^{-1}(L) \rightarrow L (\subset \mathbb{R})$, which
 is naturally regarded as a Morse function.

\ \ \ For a round fold map $f$ which is not a normal form, we can consider similar objects. We call
 a bundle naturally corresponding to the surrounding bundle of a normal form of $f$ a {\it surrounding bundle} of $f$.
%
%
%
\ \ \ We can define the following condition for a round fold map.

\begin{Def}
\label{def:2}
Let $f:M \rightarrow {\mathbb{R}}^n$ ($n \geq 2$) be a round fold map.   
If a surrounding bundle of $f$ as above 
 is a trivial bundle, then
 $f$ is said to be {\it $C^{\infty}$ trivial}. \\
\end{Def}

We introduce a fundamental example of round fold maps.

\begin{Ex}
\label{ex:1}
 Let $m,n$ be integers such that $m \geq n \geq 2$ holds. Then, by a fundamental discussion of \cite{saeki2}, a round
 fold map $f:S^m \rightarrow {\mathbb{R}}^n$ whose singular set is connected exists. The map is special generic. 
 Furthermore, any homotopy sphere of dimension $m>1$ admits a map into the plane
 as above unless $m=4$ according to a discussion in section 5 of \cite{saeki2}. Round fold maps here
 are $C^{\infty}$ trivial (see also Example 3 (1) of \cite{kitazawa2}).

\ \ \ Let $m \geq 4$ and let $n$ be an integer satisfying $m-n=1,2,3$. In section 4 of \cite{saeki2} and \cite{saeki3}, it is shown that
 if a homotopy sphere of dimension $m$ admits a special generic map into ${\mathbb{R}}^n$, then the sphere is diffeomorphic to $S^m$. Thus, on
 a homotopy sphere of dimension $m$, a round fold map into ${\mathbb{R}}^n$ whose singular set is connected exists, then the homotopy sphere is diffeomorphic to $S^m$.  
\end{Ex}

\ \ \ We easily obtain a lot of round fold maps which are $C^{\infty}$ trivial by using the following method (see also \cite{kitazawa2} and \cite{kitazawa3} for example).

\ \ \ Let $\bar{M}$ be a compact manifold with non-empty boundary $\partial \bar{M}$. Let $a \in \mathbb{R}$. Then, there exists a Morse
 function $\tilde{f}:\bar{M} \rightarrow [a,+\infty)$ satisfying the following two. 
\begin{enumerate}
\item $a$ is the minimum of $\tilde{f}$ and ${\tilde{f}}^{-1}(a)=\partial \bar{M}$ holds.
\item All the singular points of $\tilde{f}:\bar{M} \rightarrow [a,+\infty)$ are in $\bar{M}-\partial \bar{M}$ and at distinct singular points, the values are always distinct.
\end{enumerate}
Let $\Phi:\partial (\bar{M} \times \partial ({\mathbb{R}}^n-{\rm Int} D^n)) \rightarrow \partial (\partial \bar{M} \times D^n)$
 and $\phi:\partial ({\mathbb{R}}^n-{\rm Int} D^n) \rightarrow \partial D^n$ be diffeomorphisms. Let $p_1:\partial \bar{M} \times \partial ({\mathbb{R}}^n-{\rm Int} D^n) \rightarrow \partial ({\mathbb{R}}^n-{\rm Int} D^n)$ and
 $p_2:\partial \bar{M} \times \partial D^n \rightarrow \partial D^n$ be
 the canonical projections. Suppose that the following diagram commutes. 

$$
\begin{CD}
\partial \bar{M} \times \partial ({\mathbb{R}}^n-{\rm Int} D^n)  @> \Phi >> \partial \bar{M} \times \partial D^n \\
@VV p_1 V @VV p_2 V \\
\partial ({\mathbb{R}}^n-{\rm Int} D^n) @> \phi >> \partial D^n
\end{CD}
$$

By using the diffeomorphism $\Phi$, we construct $M:=(\partial \bar{M} \times D^n) {\bigcup}_{\Phi} (\bar{M} \times \partial ({\mathbb{R}}^n-{\rm Int} D^n))$. 
Let $p:\partial \bar{M} \times D^n \rightarrow D^n$ be the canonical projection. Then gluing the two
 maps $p$ and $\tilde{f} \times {{\rm id}}_{S^{n-1}}$ together by using the two diffeomorphisms $\Phi$ and $\phi$,
 we obtain a round fold map $f:M \rightarrow {\mathbb{R}}^n$.

\ \ \ If $\bar{M}$ is a compact manifold without boundary, then there exists a Morse
 function $\tilde{f}:\bar{M} \rightarrow [a,+\infty)$ such that $\tilde{f}(\bar{M}) \subset (a,+\infty)$ and that at distinct singular points, the values are always distinct. We are enough
 to consider $\tilde{f} \times {\rm id}_{S^{n-1}}$
 and embed $[a,+\infty) \times S^{n-1}$ into ${\mathbb{R}}^n$ to construct a round fold map whose source manifold is $\bar{M} \times S^{n-1}$.

\ \ \ We call this construction of a round fold map a {\it trivial spinning construction}.  

\section{New examples of round fold maps}
\label{sec:3}

In this section, we give new examples of round fold maps with their source manifolds.

\ \ \ In this section and the next section, we define a {\it trivial} embedding of the standard sphere $S^p$ ($p \geq 1$) into a smooth manifold $X$ of dimension $q>p$ without boundary as a
 smooth embedding smoothly isotopic to an embedding into a smoothly embedded open disc ${\rm Int} D^q \subset X$ which is
 unknot in the $C^{\infty}$ category. 

\begin{Prop}
\label{prop:1}
Let $M_1$ and $M_2$ be closed and connected $m$-dimensional manifolds. Assume that two round fold maps $f_1:M_1 \rightarrow {\mathbb{R}}^n$ and $f_2:M_2 \rightarrow {\mathbb{R}}^n$ {\rm (}$m > n \geq 2${\rm )} exist. We
 also assume the following two.
\begin{enumerate}
\item The fiber of a point in a proper core of $f_1$ has a connected component diffeomorphic to $S^{m-n}$. 
\item Let $C$ be the connected component of $\partial f_2(M_2)$ bounding the unbounded connected component of ${\mathbb{R}}^n-{\rm Int} f_2(M_2)$.
 Then, the embedding of ${f_2}^{-1}(C)$ into $M_2$ is a trivial embedding
 into $M_2$.
\end{enumerate}
 Then, on any manifold $M$ represented as a connected sum of the manifolds $M_1$ and $M_2$, there
 exists a round fold map $f:M \rightarrow {\mathbb{R}}^n$ satisfying the following two.
\begin{enumerate}
\item Let $P_1$ be a proper core of $f_1$. There exists an $n$-dimensional standard closed disc $Q \subset {\mathbb{R}}^n$ and
 the restriction map $f_1 {\mid}_{{f_1}^{-1}({\mathbb{R}}^n-{\rm Int} P_1)}:{f_1}^{-1}({\mathbb{R}}^n-{\rm Int} P_1) \rightarrow {\mathbb{R}}^n-{\rm Int} P_1$
 and $f {\mid}_{f^{-1}({\mathbb{R}}^n-{\rm Int} Q)}:f^{-1}({\mathbb{R}}^n-{\rm Int} Q) \rightarrow {\mathbb{R}}^n-{\rm Int} Q$ are $C^{\infty}$ equivalent.
\item Let $P_2$ be a small closed tubular neighborhood of the connected component $C$ of $f_2(S(f_2))$. Then
 ${f_2} {\mid}_{{f_2}^{-1}({\mathbb{R}}^n-{\rm Int} P_2)}$ and
 $f {\mid}_{f^{-1}(Q)}:f^{-1}(Q) \rightarrow Q$ are $C^{\infty}$ equivalent.
\end{enumerate}
\end{Prop}

\begin{proof}
This proposition is also shown in \cite{kitazawa}, \cite{kitazawa2} and \cite{kitazawa4} under the assumptions that $m \geq 2n$ holds and that
 the embedding ${f_2}^{-1}(C) \subset M_2$ is null-homotopic (and trivial as a result). We prove this proposition as the review of these proofs.

\ \ \ Let $P_1$ be a proper core of $f_1$ and $P_2$ be a small closed tubular neighborhood of the connected component $C$ of $f_2(S(f_2))$. Let
 $V_1$ be a connected component of ${f_1}^{-1}(P_1)$ such that $f_1 {\mid}_{V_1}:V_1 \rightarrow P_1$ gives the structure of a trivial $S^{m-n}$-bundle over
 $P_1$ and $V_2:={f_2}^{-1}(P_2)$. $V_2$ is a closed tubular neighborhood of ${f_2}^{-1}(C) \subset M_2$ and $V_2$ has the structure of
 a trivial linear $D^{m-n+1}$-bundle over $C$ by the assumptions that $C$ is the connected component of $\partial f_2(M_2)$ bounding the unbounded connected component of ${\mathbb{R}}^n-{\rm Int} f_2(M_2)$ and the image
 of a connected component of the singular set consisting of singular points of index $0$ and
 that the embedding ${f_2}^{-1}(C) \subset M_2$ is a trivial embedding
 into $M_2$. Note also that $f_2 {\mid}_{\partial V_2}$ gives
 the structure of a subbundle of the bundle.

\ \ \ For any diffeomorphism
 $\Psi:\partial D^m \rightarrow \partial D^m$ extending to a diffeomorphism on $D^m$ or from $M_2-(M_2-D^m)$
 onto $M_1-(M_1-D^m)$, we may ragard that there exists a diffeomorphism $\Phi:\partial V_2 \rightarrow \partial V_1$ regarded as a bundle
 isomorphism between the two trivial $S^{m-n+1}$-bundles over the ($n-1$)-dimensional standard spheres inducing a diffeomorphism
 between the base spaces and that the following relation holds, where for two smooth manifolds $X_1$ and $X_2$, $X_1 \cong X_2$ means
 that $X_1$ and $X_2$ are diffeomorphic.

\begin{eqnarray*}
& & (M_1-{\rm Int} V_1) {\bigcup}_{\Phi} (M_2-{\rm Int} V_2) \\
& \cong & (M_1-{\rm Int} V_1) {\bigcup}_{\Phi} ((D^m-{\rm Int} V_2) \bigcup (M_2-{\rm Int} D^m)) \\
& \cong & (M_1-{\rm Int} V_1) {\bigcup}_{\Phi} ((S^m-({\rm Int} V_2 \sqcup {\rm Int} D^m)) {\bigcup}_{\Psi} (M_2-{\rm Int} D^m)) \\
& \cong & (M_1-{\rm Int} D^m) {\bigcup}_{\Psi} (M_2-{\rm Int} D^m)
\end{eqnarray*}

\ \ \ This means that the resulting manifold is represented as a connected sum $M$ of $M_1$ and $M_2$ and that $M$ admits a round fold map $f:M \rightarrow {\mathbb{R}}^n$. More precisely, $f$
 is obtained by gluing the two maps ${f_1} {\mid}_{M_1-{\rm Int} V_1}$ and ${f_2} {\mid}_{M_2-{\rm Int} V_2}$. We also note that we can realize each connected sum
 of the manifolds $M_1$ and $M_2$ as the resulting manifold $M$ and that we obtain a round fold map $f:M \rightarrow {\mathbb{R}}^n$ satisfying the assumption.
\end{proof}


We call the operation of obtaining the map $f$ from the pair $(f_1,f_2)$ in the proof a canonical combining operation to the pair $(f_1,f_2)$.

\begin{Cor}
\label{cor:1}
Let $M_1$ be a closed and connected manifold of dimension $m$ and $M_2$ be a homotopy sphere of dimension $m$. Let there exist a round fold map $f_1:M_1 \rightarrow {\mathbb{R}}^n$ {\rm (}$n \geq 2${\rm )} such that the fiber
 of a point in a proper core of $f_1$ has a connected component diffeomorphic to $S^{m-n}$ and
 a round fold map $f_2:M_2 \rightarrow {\mathbb{R}}^n$. We also assume that $3n<2m$ holds.

\ \ \ Then, on any manifold $M$ represented as a connected sum of the manifolds $M_1$ and $M_2$, we
 can obtain a round fold map $f:M \rightarrow {\mathbb{R}}^n$ by a canonical combining operation to the pair $(f_1,f_2)$.  
\end{Cor}

\begin{proof}
Since the inequality $3n=3\{(n-1)+1\}<2m$ holds, from the theory of \cite{haefliger} or \cite{haefliger2}, it follows that two embeddings of $S^{n-1}$ into $M_2$
 are always smoothly isotopic. We may apply Proposition \ref{prop:1} to complete the proof.
\end{proof}

%

We have the following corollary.  

\begin{Cor}
\label{cor:2}
Let $m, n \in \mathbb{N}$, $n \geq 2$ and $3n<2m$.
If a homotopy sphere $\Sigma$ of dimension $m$ is represented as a connected sum
 of a finite number of homotopy spheres
 having the structures of $S^{m-n}$-bundles over $S^n$,
 then there exists a round fold map $f:\Sigma \rightarrow {\mathbb{R}}^n$ such that regular fibers are disjoint unions of finite copies of $S^{m-n}$ and that the number of
 connected components of the singular set and the number of connected components of the fiber of a point
 in a proper core agree.

\end{Cor}
\begin{proof}
In the situation of Corollary \ref{cor:1}, we consider round fold maps
 from $m$-dimensional homotopy spheres having the structures of $S^{m-n}$-bundles over $S^n$ into ${\mathbb{R}}^n$ as presented in Example \ref{ex:2} later. The singular set of each map has $2$ connected components and
 the fiber of a point in a proper core of each map is a disjoint union of two copies of $S^{m-n}$. By virtue of Proposition \ref{prop:1}, by using canonical combining operations inductively, we obtain
 a desired round fold map $f:\Sigma \rightarrow {\mathbb{R}}^n$.
\end{proof}

It is well-known that if an $m$-dimensional sphere has the structure of a linear bundle over an $n$-dimensional
 sphere whose fiber is an ($m-n$)-dimensional sphere, then $(m,n)=(3,2),(7,4),(15,8)$ must hold. We note that $3n<2m$ holds for $(m,n)=(7,4)$ and $(m,n)=(15,8)$.

\ \ \ It is also known that $S^3$, $S^7$ and $S^{15}$ have the structures of linear bundles over $S^2$, $S^4$ and $S^8$ whose fibers are $S^1$, $S^3$
 and $S^7$, respectively. In \cite{eellskuiper} and \cite{milnor}, there are
 some examples of $7$-dimensional homotopy spheres not
 diffeomorphic to $S^7$ having the structures of linear $S^3$-bundles over $S^4$.

\ \ \ We have the following theorem.
 
\begin{Thm}
\label{thm:1}
Every homotopy sphere of dimension $7$ admits a round fold map into ${\mathbb{R}}^4$ such that regular fibers are disjoint unions of finite copies of $S^3$ and that the number of
 connected components of the singular set and the number of connected components of the fiber of a point
 in a proper core agree. 
\end{Thm} 

\begin{proof}
By virtue of the theory of \cite{kervairemilnor} and \cite{milnor}, every homotopy sphere of dimension $7$ is
 represented as a connected sum of a finite number of oriented $7$-dimensional homotopy
 spheres admitting the structures of linear $S^3$-bundles over $S^4$. In fact, we can choose an oriented homotopy sphere of dimension $7$ so that it is a
 generator of the oriented h-cobordism group of $7$-dimensional smooth homotopy spheres. From Corollary \ref{cor:2}, the result follows.
\end{proof}

\ \ \ Theorem \ref{thm:1} states that every homotopy sphere of dimension $7$ admits a round fold map
 into ${\mathbb{R}}^4$, although a homotopy sphere of dimension $7$ not diffeomorphic to $S^7$ does not admit a round fold map
 into ${\mathbb{R}}^4$ whose singular set is connected as in Example \ref{ex:1}.

\ \ \ We also note that all the homotopy spheres admit fold maps into Euclidean spaces whose dimensions are not larger
 than those of the source manifolds (\cite{eliashberg} and \cite{eliashberg2}) and that it has been difficult to
 construct explicit examples of such fold maps as menitoned in the introdution. Theorem \ref{thm:1} gives explicit examples.

\ \ \ As another easy application of constructions performed in the proof of Proposition \ref{prop:1}, we have the following theorem.

\begin{Thm}
\label{thm:2}
Let $m,n \in \mathbb{N}$ and let $m>n \geq 2$. Let $M_i$ be a closed and connected $m$-dimensional manifold admitting
 a round fold map $f_i:M_i \rightarrow {\mathbb{R}}^n$ {\rm (}$i=1,2${\rm )}. We also asssume that the fiber of a point in a proper core
 of $f_i$ {\rm (}$i=1,2${\rm )} has a connected component diffeomorphic to the {\rm (}$m-n${\rm )}-dimensional standard sphere $S^{m-n}$.  
 Then, on any manifold $M$ represented as a connected sum of the manifolds $M_1$ and $M_2$, there exists a round fold map $f:M \rightarrow {\mathbb{R}}^n$.  
\end{Thm}
\begin{proof}
By a trivial spinning construction, we can construct a $C^{\infty}$ trivial round fold map as in the following.
There exists a Morse function $\tilde{f}:D^{m-n+1} \rightarrow [a,+\infty)$ satisfying the following three. 
\begin{enumerate}
\item $a$ is the minimum of $\tilde{f}$ and ${\tilde{f}}^{-1}(a)=\partial D^{m-n+1}$ holds.
\item $\tilde{f}$ has at least two singular points of index $0$ at which the values of the functions are local maxima.
\item All the singular points of $\tilde{f}:\partial D^{m-n+1} \rightarrow [a,+\infty)$ are in the interior ${\rm Int} D^{m-n+1}$ of the disc $D^{m-n+1}$ and at distinct singular points, the values
 are always distinct.
\end{enumerate}

Let $p:\partial D^{m-n+1} \times \partial D^n \rightarrow \partial D^n$ be
 the canonical projection. The following diagram commutes. 

$$
\begin{CD}
\partial D^{m-n+1} \times \partial ({\mathbb{R}}^n-{\rm Int} D^n)  @> {\phi} \times {\rm id}_{\partial ({\mathbb{R}}^n-{\rm Int} D^n)} >> \partial D^{m-n+1} \times \partial D^n \\
@VV {\tilde{f}} {\mid}_{\partial D^{m-n+1}} \times {\rm id}_{\partial ({\mathbb{R}}^n-{\rm Int} D^n)} V @VV p V \\
\{a\} \times \partial ({\mathbb{R}}^n-{\rm Int} D^n) @> \phi >> \partial D^n
\end{CD}
$$

By using the diffeomorphisms $\phi \times {\rm id}_{\partial ({\mathbb{R}}^n-{\rm Int} D^n)}$ and $\phi$, we
 obtain a round fold map $f_0:S^{m} \rightarrow {\mathbb{R}}^n$ such that the embedding of any connected component $C$ of the singular set $S(f_0)$ into
 $S^m$ is a trivial embedding.

\ \ \ By Proposition \ref{prop:1}, by a canonical combining operation to the pair $(f_1,f_0)$ of the maps, we can construct a new
 round fold map from $M_1$ into ${\mathbb{R}}^n$. By deforming the obtained round fold map without changing its singular sets, we obtain a round fold map ${f_1}^{\prime}:M_1 \rightarrow {\mathbb{R}}^n$ such
 that for the connected component $C^{\prime}$ of the boundary $\partial {f_1}^{\prime}(M_1)$ of ${f_1}^{\prime}(M_1)$ bounding the unbounded connected
 component of ${\mathbb{R}}^n-{\rm Int} {f_1}^{\prime}(M_1)$, ${{f_1}^{\prime}}^{-1}(C^{\prime})$
 is originally a connected component of $S(f_0) \subset S^m$ consisting of definite fold
 points. The embedding of ${{f_1}^{\prime}}^{-1}(C^{\prime})$ into $M_1$ is a trivial embedding. We may apply
 Proposition \ref{prop:1} to the pair $(f_2,{f_1}^{\prime})$ of the maps to complete the proof. 
\end{proof}

\begin{Cor}
\label{cor:3}
Let $m \geq 3$. Let $M_i$ be a closed and connected $m$-dimensional manifold admitting
 a round fold map $f_i:M_i \rightarrow {\mathbb{R}}^{m-1}$ such that $f_i(M_i)$ is diffeomorphic to $D^{m-1}$ {\rm (}$i=1,2${\rm )}. Then, on any manifold $M$ represented as a connected
 sum of the manifolds $M_1$ and $M_2$, there exists a round fold map $f:M \rightarrow {\mathbb{R}}^{m-1}$.  
\end{Cor}

\begin{proof}
Regular fibers of the maps $f_1$ and $f_2$ are always disjoint unions of circles. Thus, the statement follows from Theorem \ref{thm:2} immediately.
\end{proof}

We introduce another result.  

\begin{Thm}
\label{thm:3}
Let $m$ and $n$ be integers larger than $1$ and let $m-n \geq 1$. Then, any $m$-dimensional manifold $M$ represented
 as a connected sum of two closed and connected manifolds $M_1$ and $M_2$ satisfying the following conditions admits a round fold map $f$ into ${\mathbb{R}}^n$.
\begin{enumerate}
\item $M_1$ admits a round fold map $f_1$ whose image is diffeomorphic to $D^n$ and the fiber of a point in a proper core of which has a connected component
 diffeomorphic to $S^{m-n}$.
\item For an {\rm (}$m-n${\rm )}-dimensional closed and connected manifold $F \neq \emptyset$, $M_2$ has the structure of an $F$-bundle over $S^n$. 
\end{enumerate}

Moreprecisely, we obtain the map $f$ by a canonical combining operation to the pair $(f_1,f_2)$ of two maps where $f_2:M_2 \rightarrow {\mathbb{R}}^n$ is
 a map obtained in Theorem \ref{thm:4} by considering the bundle structure of $M_2$.
\end{Thm}

To prove Theorem \ref{thm:3}, we need the following new result. 

\begin{Thm}
\label{thm:4}
Let $F \neq \emptyset$ be a closed and connected manifold. Let $M$ be a closed manifold of dimension $m$ having the structure of an $F$-bundle
 over $S^n$ {\rm (}$m \geq n \geq 2${\rm )}. Then, $M$ admits a round fold map $f:M \rightarrow {\mathbb{R}}^n$ satisfying the following four conditions.
\begin{enumerate}
\item $f$ is $C^{\infty}$ trivial.
\item For an axis $L$ of $f$, $f^{-1}(L)$ is diffeomorphic to $F \times [-1,1]$.
\item Two connected components of the fiber of a point in a proper core of $f$ is regarded as fibers of the $F$-bundle
 over $S^n$.
\item $f(M)$ is diffeomorphic to $D^n$ and for the connected component $C:=\partial f(M)$, the embedding
 of $f^{-1}(C)$ into $M$ is a trivial embedding. More precisely, the statement in the following paragraph holds.

\ \ \ Let $P$ be a proper core of $f$. Then, as a continuous
 map into the space $f^{-1}({\mathbb{R}}^n-{\rm Int} P)$, this embedding is smoothly isotopic
 to a section of the bundle over $\partial P$ obtained by the restriction
 of the bundle given by the surjection $f {\mid}_{f^{-1}(P)}:f^{-1}(P) \rightarrow P$ extending to a section
 of the bundle $f {\mid}_{f^{-1}(P)}$.
\end{enumerate}
\end{Thm}
\begin{proof}
\ \ \ We construct a map satisfying the assumption on a $F$-bundle $M$ over $S^n$. We may represent $S^n$ as $(D^n \sqcup D^n) \bigcup (S^{n-1} \times [0,1])$, where
 we identify as $\partial (D^n \sqcup D^n)=S^{n-1} \sqcup S^{n-1}$. For a
 diffeomorphism $\Phi$ from $(S^{n-1} \sqcup S^{n-1}) \times F$ onto $(\partial D^n \sqcup \partial D^n) \times F$ which is a bundle isomorphism between
 the trivial $F$-bundles inducing the identification of the base spaces, we may
 represent $M$ as $((D^n \sqcup D^n) \times F) {\bigcup}_{\Phi} (S^{n-1} \times [0,1] \times F)=(D^n \times (F \sqcup F)) {\bigcup}_{\Phi} (S^{n-1} \times [0,1] \times F)$ since the base space of the bundle $M$ is a standard sphere. \\
\ \ \ There exists a Morse function $\tilde{f}:F \times [0,1] \rightarrow [a,+\infty)$, where $a \in \mathbb{R}$ is the minimum, as in the presentation of a trivial spinning construction before. We use a trivial spinning construction as the following. \\
\ \ \ We consider a map $\tilde{f} \times {\rm id}_{S^{n-1}}$ and the canonical projection $p:D^n \times (F \sqcup F) \rightarrow D^n$. For the maps $\Phi$, $\tilde{f} \times {\rm id}_{S^{n-1}}$ and $p$
 and a diffeomorphism $\phi:\partial ({\mathbb{R}}^n-{\rm Int} D^n) \rightarrow \partial D^n$, we may
 assume that the following diagram commutes.

$$
\begin{CD}
F \times (\{0\} \sqcup \{1\}) \times \partial ({\mathbb{R}}^n-{\rm Int} D^n) @> \Phi >> F \times (\{0\} \sqcup \{1\}) \times \partial D^n \\
@VV {\tilde{f}} {\mid}_{F \times (\{0\} \sqcup \{1\})} \times {\rm id}_{\partial ({\mathbb{R}}^n-{\rm Int} D^n)} V @VV p {\mid}_{F \times (\{0\} \sqcup \{1\}) \times \partial D^n} V \\
\{a\} \times \partial ({\mathbb{R}}^n-{\rm Int} D^n) @> \phi >> \partial D^n  
\end{CD}
$$  

Then, by gluing the maps $p$ and $\tilde{f} \times {\rm id}_{S^{n-1}}$ by the pair of diffeomorphisms $(\Phi,\phi)$, we obtain a $C^{\infty}$ trivial
 round fold map $f:M \rightarrow {\mathbb{R}}^n$. \\
\ \ \ We may assume that the embedding of $f^{-1}(\partial f(M))$ into $M$ is smoothly isotopic to an embedding $\partial D^n \times \{p\} \subset D^n \times (F \sqcup F) \subset M$. In fact, by considering the bundle structure over $S^n=(D^n \sqcup D^n) \bigcup (S^{n-1} \times [0,1])$ of $M$, we can
 easily take $\Phi$ so that this holds. This
 means that the embedding of $f^{-1}(\partial f(M))$ into $M$ is a trivial embedding into $M_2$.

\ \ \ We see that $f$ is a round fold map satisfying the given conditions. This completes the proof.
\end{proof}

\begin{Rem}
Theorem \ref{thm:4} was also shown in \cite{kitazawa4} by the author for a similar map $f$ without the last condition. Furthermore, the author has shown
 that a manifold admitting such a map has the structure of an $F$-bundle over $S^n$.
\end{Rem}

\begin{Ex}
\label{ex:2}
In the situation of the proof of Theorem \ref{thm:4}, let $F:=S^{m-n}$ ($m>n \geq 2$) and let $\tilde{f}:F \times [-1,1] \rightarrow [a,+\infty)$ be a Morse
 function with two singular points, where $a \in \mathbb{R}$ is the minimum, as in the presentation of a trivial spinning construction before (we easilly obtain such a Morse function). As a result, on any
 manifold having the structure of an $S^{m-n}$-bundle
 over $S^n$, we have a round fold map as in Theorem \ref{thm:4} whose singular set consists of two connected components and the fiber of
 a proper core of which is a disjoint union of two copies of $S^{m-n}$. The author constructed
 such a map first in \cite{kitazawa} without assuming the last condition mentioned in Theorem \ref{thm:4}. Furthermore, the author
 has shown that a manifold admitting such a map has the structure of an $S^{m-n}$-bundle over $S^n$. 
\end{Ex}

\begin{proof}[Proof of Theorem \ref{thm:3}]
 By applying Proposition \ref{prop:1} to the pair of a round fold map from $M_1$ into ${\mathbb{R}}^n$ appearing in the assumption
 and a round fold map from $M_2$ into ${\mathbb{R}}^n$ constructed in Theorem \ref{thm:4}, we obtain a round
 fold map on the manifold $M$. This completes the proof.
\end{proof}

We have the following corollary to Theorem \ref{thm:3}.

\begin{Cor}
\label{cor:4}
Let $m$ and $n$ be integers larger than $1$ and let $m-n \geq 1$. Then, any $m$-dimensional manifold $M$ represented
 as a connected sum of two closed and connected manifolds $M_1$ and $M_2$ satisfying the following conditions admits a round fold map
 $f$ into ${\mathbb{R}}^n$ such that the fiber of a point in a proper core of which consist of three connected components.
\begin{enumerate}
\item $M_1$ has the structure of an $S^{m-n}$-bundle over $S^n$.
\item For an {\rm (}$m-n${\rm )}-dimensional closed and connected manifold $F_1 \neq \emptyset$, $M_2$ has the structure of an $F_1$-bundle over $S^n$. 
\end{enumerate}

More precisely, we obtain the map $f$ by a canonical combining operation to the pair of two maps obtained in Theorem \ref{thm:4} by considering the bundle structures.

Let $F_2 \neq \emptyset$ be an {\rm (}$m-n${\rm )}-dimensional closed and connected manifold. Then, on any manifold represented as a connected sum
 of the manifold $M$ and any manifold having the structure of an $F_2$-bundle over $S^n$, we obtain a round fold map
 into ${\mathbb{R}}^n$ by a canonical combining operation to the pair of maps $f$ and a map obtained in Theorem \ref{thm:4} by considering
 the bundle structure of the latter manifold. 
\end{Cor}

We note again that $S^3$, $S^7$ and $S^{15}$ have the structures of a linear $S^1$-bundle over $S^2$, a linear $S^3$-bundle over $S^4$ and a
 linear $S^7$-bundle over $S^8$, respectively. By applying Theorem \ref{thm:3} inductively or Corollary \ref{cor:4}, we have the following theorem.

\begin{Thm}
\label{thm:5}
Let $m$ and $n$ be integers larger than $1$ and let $m-n \geq 1$.
\begin{enumerate}
\item
\label{thm:5.1}
 Any manifold represented as a connected sum of $l \in \mathbb{N}$ closed manifolds
 having the structures of $S^{m-n}$-bundles over $S^n$ admits a round fold map $f$ into ${\mathbb{R}}^n$ satisfying the following four.
\begin{enumerate}
\item All the regular fibers of $f$ are disjoint unions of finite copies of $S^{m-n}$.
\item The number of connected components of $S(f)$ and the number of connected components of the fiber of a point in a proper core of $f$ are $l$. 
\item All the connected components of the fiber of a point in a proper core of $f$ are regarded as fibers of the $S^{m-n}$-bundles
 over $S^n$ and a fiber of any $S^{m-n}$-bundle over $S^n$ appeared in the connected sum is regarded as a connected component of the
 fiber of a point in a proper core of $f$.
\item For any connected component $C$ of $f(S(f))$ and a small closed tubular neighborhood $N(C)$ of $C$ such that $\partial N(C)$ is the disjoint
 union of two connected components $C_1$ and $C_2$, $f^{-1}(N(C))$ has the structures
 of trivial bundles over $C_1$ and $C_2$ whose fiber is diffeomorphic to the {\rm (}$m-n+1${\rm )}-dimensional standard closed disc $D^{m-n+1}$ or
 a disjoint union of finite copies of the {\rm (}$m-n${\rm )}-dimensional standard sphere with the
 interior of the union of three disjoint {\rm (}$m-n+1${\rm )}-dimensional standard closed discs embedded smoothly removed and $f {\mid}_{f^{-1}(C_1)}:f^{-1}(C_1) \rightarrow C_1$ and $f {\mid}_{f^{-1}(C_2)}:f^{-1}(C_2) \rightarrow C_2$ give the structures
 of subbundles of the bundles $f^{-1}(N(C))$.
\end{enumerate}
\item
\label{thm:5.2}
 Let $(m,n)=(3,2), (7,4), (15.8)$. Then any $m$-dimensional manifold $M$ represented
 as a connected sum of two closed and connected manifolds having the structures of bundles over $S^n$ admits a round fold map
 $f$ into ${\mathbb{R}}^n$. More precisely, we obtain the map $f$ by applying Corollary \ref{cor:4} setting $M_1:=S^m$ in the situation of the corollary. 
\end{enumerate}
\end{Thm}

\begin{Thm}
\label{thm:6}
Every $7$-dimesional homotopy sphere admits a round fold map $f$ into ${\mathbb{R}}^4$ satisfying the four conditions mentioned in Theorem \ref{thm:5} and the following three hold.
\begin{enumerate}
\item A $7$-dimensional homotopy sphere $M$ admits such a round fold map into ${\mathbb{R}}^4$ such that the singular set consists of one connected
 component if and only if $M$ is diffeomorphic to the $7$-dimensional standard sphere $S^7$.  
\item A $7$-dimensional homotopy sphere $M$ admits such a round fold map into ${\mathbb{R}}^4$ such that the singular set consists of two connected
 components if and only if $M$ has the structure of an $S^3$-bundle over $S^4$. Just $16$ of all the $28$ classes of the $7$-dimensional oriented h-cobordism
 group include such homotopy spheres. 
\item Every $7$-dimensional homotopy sphere $M$ admits such a round fold map into ${\mathbb{R}}^4$ such that the singular set consists of three connected
 components.
\end{enumerate}
\end{Thm}
\begin{proof}
The former part follows from the mentioned fact that every $7$-dimesional homotopy sphere is represented as a connected sum of $7$-dimesional
 homotopy spheres admitting the structures of $S^3$-bundles over $S^4$ and Theorem \ref{thm:5}. The proof of Theorem \ref{thm:1} also
 shows this fact.

\ \ \ We prove the three statements of the latter part. The first statement is mentioned in Example \ref{ex:1}. The orientation preserving
 diffeomorphism group of $S^3$ is known
 to be homotopy equivalent to $SO(4)$ and the natural inclusion of $SO(4)$, which we regard
 as the group of all the linear transformations on $S^3$, into
 the group of the orientation preserving diffeomorphisms give homotopy equivalences (see \cite{hatcher}). Thus, a round fold map mentioned in the second
 statement is in fact regarded as a $C^{\infty}$ trivial round fold map mentioned in Example \ref{ex:2}. This completes the proof of the second part. 

\ \ \ For the last statement, we need the facts that there exists an isomorphism from the $7$-dimensional oriented h-cobordism
 group onto the cyclic group $\mathbb{Z}/28\mathbb{Z}$ and that the values of the isomorphism of $16$ classes in the second statement is $0, 1, 3, 6, 7, 8, 10, 13, 14, 15, 17, 20, 21, 22, 24$ and $27$ of \cite{eellskuiper}.
 Every $28$ element of $\mathbb{Z}/28\mathbb{Z}$ is represented as a sum of two elements of these $16$ values. This means that every $7$-dimensional homotopy sphere $M$ is represented as
 a connected sum of two homotopy spheres in the second statement. From Theorem \ref{thm:5} (\ref{thm:5.1}), this completes the proof of the last statement. 
\end{proof}

\ \ \ We also have the following proposition. 

\begin{Prop}
\label{prop:2}
Let $M_1$ and $M_2$ be closed and connected $m$-dimensional manifolds.

\ \ \ Assume that there exists a round fold map $f_1:M_1 \rightarrow {\mathbb{R}}^n$ {\rm (}$m \geq n \geq 2${\rm )} and that $m>n$ holds. Assume also that
 there exists a connected component of the
 fiber of a point in a proper core of $f_1$ diffeomorphic to $S^{m-n}$ and that the embedding of
 the connected component into $M_1$ is a trivial embedding. Furthermore, we also assume
 the existence of a round fold map $f_2:M_2 \rightarrow {\mathbb{R}}^n$ satisfying the following two.
\begin{enumerate}
\item For a small closed tubular neighborhood
 $N(C)$ of the connected component $C$ of $\partial f_2(M_2)$ bounding the unbounded connected
 component of ${\mathbb{R}}^n-{\rm Int} f_2(M_2)$, ${f_2}^{-1}(N(C))$ has the structure
 of a trivial $D^{m-n+1}$-bundle over the connected component
 $C^{\prime}$ of $\partial N(C)$ in $f_2(M_2)$.
\item $f_2 {\mid}_{{f_2}^{-1}({C}^{\prime})}$ gives the structure of a subbundle of the previous bundle ${f_2}^{-1}(N(C))$ over $C^{\prime}$.
\end{enumerate}
Then, on any manifold $M$ represented as a connected sum of the manifolds $M_1$ and $M_2$, we obtain a round fold map $f:M \rightarrow {\mathbb{R}}^n$ by a {\it canonical
 combining operation} to the pair $(f_1,f_2)$ of the maps defined in the same manner.  
\end{Prop}

\begin{proof}
We can prove this proposition by the same construction as that of the proof of Proposition \ref{prop:1}. However, we present
 the construction again. \\ 
\ \ \ Let $P_1$ be a proper core of $f_1$ and $V_1$ be a connected component of ${f_1}^{-1}(P_1)$ such that for $p \in P_1$, the embedding of ${f_1}^{-1}(p) \bigcap V_1$ into $M_1$ is
 a trivial embedding. We note that $f_1 {\mid}_{V_1}:V_1 \rightarrow P_1$ gives the
 structure of a trivial bundle. Let $P_2:=N(C)$ and $V_2:={f_2}^{-1}(P_2)$.

\ \ \ Similarly to the proof of Proposition \ref{prop:1}, for any diffeomorphism
 $\Psi:\partial D^m \rightarrow \partial D^m$ extending to a diffeomorphism on $D^m$ or from $M_2-(M_2-D^m)$
 onto $M_1-(M_1-D^m)$, we may ragard that there exists a diffeomorphism $\Phi:\partial V_2 \rightarrow \partial V_1$ regarded as a bundle
 isomorphism between the two trivial $S^{m-n+1}$-bundles over the ($n-1$)-dimensional standard spheres inducing a diffeomorphism
 between the base spaces and that the following relation holds, where for two smooth manifolds $X_1$ and $X_2$, $X_1 \cong X_2$ means
 that $X_1$ and $X_2$ are diffeomorphic as in the proof of Proposition \ref{prop:1}. \\ 

\begin{eqnarray*}
 & & (M_1-{\rm Int} V_1) {\bigcup}_{\Phi} (M_2-{\rm Int} V_2) \\
 & \cong & (M_1-{\rm Int} V_1) {\bigcup}_{\Phi} (M_2-{\rm Int} V_2) \\
 & \cong & ((M_1-{\rm Int} D^m) {\bigcup} (D^m-{\rm Int} V_1)) {\bigcup}_{\Phi} (M_2-{\rm Int} V_2) \\
 & \cong & ((M_1-{\rm Int} D^m) {\bigcup}_{\Psi} (S^m-({\rm Int} D^m \sqcup {\rm Int} V_1))) {\bigcup}_{\Phi} (M_2-{\rm Int} V_2) \\
 & \cong & (M_1-{\rm Int} D^m) {\bigcup}_{\Psi} (M_2-{\rm Int} D^m)  
\end{eqnarray*}

\ \ \ This means that the resulting manifold is a connected sum $M$ of the manifolds $M_1$ and $M_2$ and that $M$ admits a round fold map $f:M \rightarrow {\mathbb{R}}^n$. More precisely, $f$
 is obtained by gluing the two maps ${f_1} {\mid}_{M_1-{\rm Int} V_1}$ and ${f_2} {\mid}_{M_2-{\rm Int} V_2}$. We also note that we can realize each connected sum
 of the manifolds $M_1$ and $M_2$ as the resulting manifold $M$ and that we obtain a round fold map $f:M \rightarrow {\mathbb{R}}^n$.
\end{proof}

For example, we have the following theorem.

\begin{Thm}
\label{thm:7}
If a closed and connected manifold $M$ of dimension $7$ admits a round fold map into ${\mathbb{R}}^4$, then for
 any homotopy sphere $\Sigma$ of dimension $7$, any manifold represented as a connected sum of $M$ and $\Sigma$ admits a round fold map into ${\mathbb{R}}^4$  
\end{Thm}

\begin{proof}
Any linear $D^4$-bundle over $S^3$ is trivial since ${\pi}_2(SO(4)) \cong {\pi}_2(S^3) \oplus {\pi}_2(SO(3)) \cong \{0\}$ holds. Let
 $f:M \rightarrow {\mathbb{R}}^4$ be a round fold map. For a small closed tubular neighborhood
 $N(C)$ of the connected component $C$ of $\partial f(M)$ bounding the unbounded connected
 component of ${\mathbb{R}}^4-{\rm Int} f(M)$, ${f}^{-1}(N(C))$ has the structure
 of a trivial linear $D^4$-bundle over the connected component
 $C^{\prime}$ of $\partial N(C)$ in $f(M)$ and $f {\mid}_{f^{-1}({C}^{\prime})}$ gives the structure of a subbundle of the bundle $f^{-1}(N(C))$. Thus, $f$ satisfies the conditions posed on the map $f_2$ in Proposition \ref{prop:2}. \\
\ \ \ Two smooth embeddings of $S^3$ into a $7$-dimensional homotopy sphere are always smoothly isotopic from the theory of \cite{haefliger} or \cite{haefliger2} with the inequality $3 \times (3+1)=12<2 \times 7=14$. We may apply Proposition \ref{prop:2} to
 a round fold map in Theorem \ref{thm:1} and $f:M \rightarrow {\mathbb{R}}^4$ to construct a desired round fold map. This completes the proof.
\end{proof}

\begin{Rem}
If a closed and connected manifold $M$ of dimension $7$ admits a round fold map into ${\mathbb{R}}^2$, then a result similar
 to Theorem \ref{thm:7} holds. In fact, by \cite{saeki2} or Example \ref{ex:1}, every homotopy sphere of dimension $7$ admits a round fold map whose singular set is connected into the plane and
 we only consider a connected sum of this map and the given round fold map (for a connected sum of such
 maps, for example, see section 5 of \cite{saeki2}, in which a connected sum of two special generic maps into
 the plane was introduced). 
 For the case where $n=3,5,6,7$ holds, we don't know whether a result similar to Theorem \ref{thm:7} holds for
 round fold maps into ${\mathbb{R}}^n$.
\end{Rem}

We also have the following theorem.

\begin{Thm}
\label{thm:8}
\begin{enumerate}
\item
\label{thm:8.1}
 Let $M_1$ and $M_2$ be closed and connected $3$-dimensional manifolds. Let $M_1$ admit a round fold map $f_1:M_1 \rightarrow {\mathbb{R}}^2$ whose image $f_1(M_1)$ is diffeomorphic to $D^2$. Let $M_2$ admit a round
 fold map $f_2:M_2 \rightarrow {\mathbb{R}}^2$ such that for the connected component $C$ of the
 boundary $\partial f_2(M_2)$ of the image $f_2(M_2)$ bounding the unbounded connected component
 of ${\mathbb{R}}^n-{\rm Int} f_2(M_2)$, the 1st Stiefel-Whitney class of the manifold $M_2$ vanishes on the cycle
 represented by a circle ${f_2}^{-1}(C)$.
Then, any manifold $M$ represented as a connected sum of the manifolds $M_1$ and $M_2$ admits a round fold map.
\item
\label{thm:8.2}
 Let $M_1$ and $M_2$ be closed and connected $7$-dimensional manifolds and let $M_1$ admit
 a round fold map $f_1:M_1 \rightarrow {\mathbb{R}}^4$ such that the fiber of a point in a proper core of $f_1$ has a connected component diffeomorphic to $S^3$. Then, any
 manifold $M$ represented as a connected
 sum of the manifolds $M_1$ and $M_2$ admit a round fold map. 

\item
\label{thm:8.3}
 Let $M_1$ and $M_2$ be closed and connected $15$-dimensional manifolds. Let $M_1$ admit
 a round fold map $f_1:M_1 \rightarrow {\mathbb{R}}^8$ admit a round fold map such that the fiber of a point in a proper core of $f_1$ has a connected component diffeomorphic to $S^7$. Furthermore, let $M_2$ admit a round
 fold map $f_2:M_2 \rightarrow {\mathbb{R}}^2$ such that for the connected component $C$ of the
 boundary $\partial f_2(M_2)$ of the image $f_2(M_2)$ bounding the unbounded connected component
 of ${\mathbb{R}}^n-{\rm Int} f_2(M_2)$ and a closed tubular neighborhood $N(C)$ of $C$, $f^{-1}(N(C))$ has
 the structure of a trivial linear $D^8$-bundle over $C$ and that $f {\mid}_{\partial f^{-1}(N(C))}:\partial f^{-1}(N(C)) \rightarrow C$ gives the structure of its subbundle. Then, any
 manifold $M$ represented as a connected
 sum of the manifolds $M_1$ and $M_2$ admit a round fold map. 
\end{enumerate} 
\end{Thm}
\begin{proof}
We prove (\ref{thm:8.1}). As in the proof of Theorem \ref{thm:4}, we can construct a round
 fold map $f_0:S^3 \rightarrow {\mathbb{R}}^2$. The pair of the map $f_0$ and the map $f_2$ satisfies the assumption
 of Proposition \ref{prop:2}. In fact, the embeddings into $S^3$ of both connected components of the fiber of a point in a proper core of former map $f_0$ are
 trivial embeddings and the assumption on the 1st Stiefel-Whitney class of $M_2$ means that for a small closed tubular
 neighborhood $N(C)$ of $C$, ${f_2}^{-1}(N(C))$ has the structure of a trivial linear $D^2$-bundle over $C$ and
 that $f_2 {\mid}_{\partial {f_2}^{-1}(N(C))}:\partial {f_2}^{-1}(N(C)) \rightarrow C$ gives the structure of its subbundle. By applying
 Proposition \ref{prop:2}, we obtain a new round fold map ${f_2}^{\prime}:M_2 \rightarrow {\mathbb{R}}^2$. Moreover, we can construct
 the map ${f_2}^{\prime}$ so that the pair of the maps $f_1$ and ${f_2}^{\prime}$ satisfies the assumption
 of Proposition \ref{prop:1} by the last condition of the resulting map in Theorem \ref{thm:4}. By applying Proposition \ref{prop:1}, we obtain a desired round fold map from $M$ into ${\mathbb{R}}^2$.

\ \ \ We discuss (\ref{thm:8.2}). $S^7$ has the structure of an $S^3$-bundle over $S^4$ and as
 in the proof of Theorem \ref{thm:4}, we can construct a round
 fold map $f_0:S^7 \rightarrow {\mathbb{R}}^4$. For the connected component $C$ of the
 boundary $\partial f_2(M_2)$ of the image $f_2(M_2)$ bounding the unbounded connected component
 of ${\mathbb{R}}^n-{\rm Int} f_2(M_2)$ and for a small closed tubular
 neighborhood $N(C)$ of $C$, $f^{-1}(N(C))$ has the structure of a trivial linear $D^4$-bundle over $C$ by the fact that ${\pi}_2(SO(4)) \cong \{0\}$ holds and
 that $f {\mid}_{\partial f^{-1}(N(C))}:\partial f^{-1}(N(C)) \rightarrow C$ gives the structure of its subbundle. These two facts mean that we can prove (\ref{thm:8.2}) similarly.

\ \ \ We can prove (\ref{thm:8.3}) similarly and this completes the proof of all the statements.
\end{proof}

\begin{Thm}
\label{thm:9}
\begin{enumerate}
\item
\label{thm:9.1}
 Let $M_1$ and $M_2$ be closed and connected $3$-dimensional manifolds. For $i=1,2$, let $M_i$ admit a round
 fold map $f_i:M_1 \rightarrow {\mathbb{R}}^2$ such that for the connected component $C_i$ of the
 boundary $\partial f_i(M_i)$ of the image $f_i(M_i)$ bounding the unbounded connected component
 of ${\mathbb{R}}^n-{\rm Int} f_i(M_i)$, the 1st Stiefel-Whitney class of the manifold $M_i$ vanishes on the cycle
 represented by a circle ${f_i}^{-1}(C_i)$.
Then, any manifold $M$ represented as a connected sum of the manifolds $M_1$ and $M_2$ admits a round fold map into ${\mathbb{R}}^2$ again.
\item
\label{thm:9.2}
 Let $M_1$ and $M_2$ be closed and connected $7$-dimensional manifolds admitting round fold maps into ${\mathbb{R}}^4$. Then, any
 manifold $M$ represented as a connected
 sum of the manifolds $M_1$ and $M_2$ admits a round fold map into ${\mathbb{R}}^4$ again. 

\item
\label{thm:9.3}
 Let $M_1$ and $M_2$ be closed and connected $15$-dimensional manifolds. For $i=1,2$, let $M_i$ admit a round
 fold map $f_i:M_1 \rightarrow {\mathbb{R}}^8$ such that for the connected component $C_i$ of the
 boundary $\partial f_i(M_i)$ of the image $f_i(M_i)$ bounding the unbounded connected component
 of ${\mathbb{R}}^8-{\rm Int} f_i(M_i)$ and a closed tubular neighborhood $N(C_i)$, ${f_i}^{-1}(N(C_i))$ has the structure of a
 trivial linear $D^8$-bundle over $C_i$ and that $f {\mid}_{\partial f^{-1}(N(C_i))}:\partial f^{-1}(N(C_i)) \rightarrow C_i$ gives
 the structure of its subbundle. Then, any manifold $M$ represented as a connected
 sum of the manifolds $M_1$ and $M_2$ admits a round fold map into ${\mathbb{R}}^8$ again.
\end{enumerate} 
\end{Thm}

\begin{proof}
We prove Theorem \ref{thm:9} (\ref{thm:9.1}). As in the proof of Theorem \ref{thm:4}, we can construct a round
 fold map $f_0:S^3 \rightarrow {\mathbb{R}}^2$. The pair of the map $f_0$ and the map $f_i$ satisfies the assumption
 of Proposition \ref{prop:2}, which follows from a discussion in the proof of Theorem \ref{thm:8} (\ref{thm:8.1}), and by applying
 Proposition \ref{prop:2}, we obtain a new round fold map ${f_i}^{\prime}:M_i \rightarrow {\mathbb{R}}^2$. Moreover, we can construct
 the map ${f_i}^{\prime}$ so that the pair of the maps ${f_1}^{\prime}$ and ${f_2}^{\prime}$ satisfies the assumption
 of Proposition \ref{prop:1} by the last condition of the resulting map in Theorem \ref{thm:4}. By applying
 Proposition \ref{prop:1}, we obtain a desired round fold map from the manifold $M$ represented as a connected sum of the two manifolds $M_1$ and $M_2$ into ${\mathbb{R}}^2$.

\ \ \ We can easily prove other statements similarly.
\end{proof}

\begin{Rem}
We do not know whether we can prove arguments similar to Theorems \ref{thm:8} and \ref{thm:9} for other pairs $(m,n)$ ($m>n\geq 2$) of dimensions.
\end{Rem}

\section{Constructions of $C^{\infty}$ trivial maps}
\label{sec:4}

We show the following theorem.

\begin{Thm}
\label{thm:10}
Let $M_1$ and $M_2$ be closed and connected $m$-dimensional manifolds. Assume that two $C^{\infty}$ trivial round fold maps $f_1:M_1 \rightarrow {\mathbb{R}}^n$
 and $f_2:M_2 \rightarrow {\mathbb{R}}^n$ {\rm (}$m>n \geq 2${\rm )} exist. We
 also assume the following conditions.
\begin{enumerate}
\item
\label{thm:10.1}
 The fiber of a point in a proper core of $f_1$ has a connected component diffeomorphic to $S^{m-n}$. 
\item
\label{thm:10.2}
 Isomorphisms on the trivial $S^{m-n}$-bundle over $S^{n-1}$ inducing the identity map on the base space $S^{n-1}$ are always smoothly isotopic to the identity map
 on the total space of the trivial bundle if they are orientation preserving diffeomorphisms on the total space.
\item
\label{thm:10.3}
 At least one of the following two holds.
\begin{enumerate}
\item
\label{thm:10.3.1}
 Let $C$ be the connected component of $\partial f_2(M_2)$ bounding the unbounded connected component of ${\mathbb{R}}^n-{\rm Int} f_2(M_2)$.
 Then, the embedding of ${f_2}^{-1}(C)$ into $M_2$ is a trivial embedding
 into $M_2$.
\item
\label{thm:10.3.2}
 The embedding of a connected component of the fiber of a point in a proper core of $f_1$ diffeomorphic to $S^{m-n}$ into
 $M_1$ is a trivial embedding. 
\end{enumerate}
\end{enumerate}
 Then, on any manifold $M$ represented as a connected sum of the manifolds $M_1$ and $M_2$, we obtain
 a $C^{\infty}$ trivial round fold map $f:M \rightarrow {\mathbb{R}}^n$ by a canonical combining operation to the pair $(f_1,f_2)$ of maps.
\end{Thm}

\begin{proof}
The condition (\ref{thm:10.3.1}) or (\ref{thm:10.3.2}) is assumed. Thus, by the assumption (\ref{thm:10.1}), we may apply
 the method of the proof of Proposition \ref{prop:1} (resp. \ref{prop:2}). We abuse notation in the proof of these Propositions such as manifolds $V_1$ and $V_2$, whose
 boundaries $\partial V_1$ and $\partial V_2$ have the structures of trivial $S^{m-n}$-bundles over $S^{n-1}$, and
 an isomorphism $\Phi:\partial V_2 \rightarrow \partial V_1$ between the bundles.

\ \ \ We can take an isomorphism $\Phi$ between
 the two $S^{m-n}$-bundles over $S^{n-1}$ inducing
 a diffeomorphism between the base spaces as in these proofs. Furthermore, for any
 diffeomorphism between the base spaces, we can take $\Phi$ inducing the diffeomorphism
 and by the assumption (\ref{thm:10.2}), such isomorphisms are always smoothly isotopic to
 the product of the diffeomorphism between
 the base spaces and a diffeomorphism between the fibers, which extends to a diffeomorphism between two standard closed
 discs of dimension $m-n+1$ (we regard the fibers as the boundaries of the standard closed discs here). 

\ \ \ By the constructions of the manifolds $M$ and maps $f$ in these proofs, we can realize any connected
 sum of the manifolds $M_1$ and $M_2$ as the resulting source manifold and we can take a diffeomorphism between
 the base space of the bundles $\phi$ and an isomorphism $\Phi:\partial V_2 \rightarrow \partial V_1$ between the bundles inducing the
 diffeomorphism $\phi$ so that the resulting map is a $C^{\infty}$ trivial round fold map from the connected sum $M$ into ${\mathbb{R}}^n$.   
\end{proof}

As specific cases, we have the following theorems.

\begin{Thm}
\label{thm:11}
Let $M_1$ and $M_2$ be closed and connected $m$-dimensional manifolds. Assume that two $C^{\infty}$ trivial round fold maps $f_1:M_1 \rightarrow {\mathbb{R}}^n$
 and $f_2:M_2 \rightarrow {\mathbb{R}}^n$ {\rm (}$m > n \geq 2${\rm )} exist. We also assume that $f_2(M_2)$ is diffeomorphic to $D^n$. 
Furthermore, suppose that one of the following two hold.
\begin{enumerate}
\item $n \geq 3$ and $m-n=1$.
\item $(m,n)=(5,3)$ or $(m,n)=(6,3)$ holds and the fibers of points in proper cores of $f_1$ and $f_2$ have connected components diffeomorphic to the standard sphere $S^{m-n}$.
\end{enumerate}

\ \ \ Then, on any manifold $M$ represented as a connected sum of the manifolds $M_1$ and $M_2$, we obtain a $C^{\infty}$ trivial
 round fold map $f:M \rightarrow {\mathbb{R}}^n$ by using a canonical combining operation to the pair $(f_1,f_2)$ of maps. 
\end{Thm}

\begin{proof}
 Let $Q$ be a proper core of $f_2$.
 
\ \ \ In the first case, regular fibers are always disjoint unions of finite copies of $S^1$. Note that the group
 of diffeomorphisms consisting of all the orientation preserving diffeomorphisms on $S^1$ has the same
 homotopy type as that of $SO(2)$ and $S^1$ and that the natural inclusion of $SO(2)$, which we regard as the group of all the linear transformations on $S^1$, into
 the group of the diffeomorphism gives a homotopy equivalence.
 For the connected component $C$ of $\partial f_2(M_2)$ bounding
 the unbounded connected component of ${\mathbb{R}}^n-{\rm Int} f_2(M_2)$, the
 embedding of the connected component ${f_2}^{-1}(C)$ of the singular set $S(f_2)$ into $M_2$ is smoothly isotopic to every
 section of the trivial bundle given by the submersion $f {\mid}_{{f_2}^{-1}(\partial Q)}:{f_2}^{-1}(\partial Q) \rightarrow \partial Q$ as a map into $M_2$
 by the assumptions that ${\pi}_{n-1}(S^1)$ is zero ($n \geq 3$ is assumed) and that $f_2$ is $C^{\infty}$ trivial.
 Then, the embedding of ${f_2}^{-1}(C)$ into $M_2$ is a trivial embedding
 into $M_2$ since the bundle given by the submersion $f_2 {\mid}_{{f_2}^{-1}(Q)}:{f_2}^{-1}(Q) \rightarrow Q$ is a trivial bundle over a standard closed disc whose fiber is diffeomorphic to a disjoint union of finite copies of $S^1$.

\ \ \ In the second case, the fiber of a point in a proper core of $f_2$ has a connected component diffeomorphic to the standard sphere $S^2$ or $S^3$.
 Note that the orientation preserving diffeomorphism group of $S^k$ (k=2,3) has the same
 homotopy type as that of $SO(k+1)$, and that the natural inclusion of $SO(k+1)$, which we regard
 as the group of all the linear transformations on $S^k$, respectively, into
 the orientation preserving diffeomorphism group gives a homotopy equivalence (see \cite{smale} for the $k=2$ case and see \cite{hatcher} for the $k=3$ case as mentioned in the proof of Theorem \ref{thm:6}). We
 also note that the groups ${\pi}_{n-1}(SO(3)) \cong {\pi}_2(SO(3))$ and ${\pi}_{n-1}(SO(4)) \cong {\pi}_2(SO(4))$ are
 zero. We obtain a fact similar to one in the first case. 

\ \ \ Thus, in both cases, all the assumptions of Theorem \ref{thm:10} are satisfied. This completes the proof. 
\end{proof}

\begin{Thm}
\label{thm:12}
Let $M_1$ and $M_2$ be closed and connected $m$-dimensional manifolds. Assume that two $C^{\infty}$ trivial round fold maps $f_1:M_1 \rightarrow {\mathbb{R}}^n$
 and $f_2:M_2 \rightarrow {\mathbb{R}}^n$ {\rm (}$m>n \geq 2${\rm )} exist. We
 also assume the following two.
\begin{enumerate}
\item
\label{thm:12.1}
 The fibers of points in proper cores of $f_1$ and $f_2$ have connected components diffeomorphic to $S^{m-n}$. 
\item
\label{thm:12.2}
 Isomorphisms on the trivial $S^{m-n}$-bundle over $S^{n-1}$ inducing the identity map on the base space $S^{n-1}$ are always smoothly isotopic to the identity map
 on the total space of the trivial bundle if they are orientation preserving diffeomorphisms on the total space.
\end{enumerate}
 Then, on any manifold $M$ represented as a connected sum of the manifolds $M_1$ and $M_2$, by a canonical combining operation to the pair $(f_1,f_2)$ of the maps, we obtain a $C^{\infty}$ trivial round fold
 map $f:M \rightarrow {\mathbb{R}}^n$ such that the fiber of a point in a proper core of $f$ has a connected component
 diffeomorphic to the standard sphere $S^{m-n}$.
\end{Thm}
\begin{proof}
By the assumption (\ref{thm:12.2}), for the boundary $C$ of the image $f_2(M)$ of $f_2$, which is diffeomorphic to $D^n$ by the
 assumption (\ref{thm:12.1}) and the assumption that $M_2$ is connected, by using a method similar
 to that of the proof of Theorem \ref{thm:11}, we can show that the embedding of ${f_2}^{-1}(C)$ into $M_2$
 is a trivial embedding
 into $M_2$. Thus, all the assumptions of Theorem \ref{thm:10} are satisfied. This completes the proof.
\end{proof}

\begin{Rem}
In the situation of Theorem \ref{thm:12}, we can obtain a $C^{\infty}$ trivial round fold map on any manifold represented as a
 connected sum $M$ of the manifolds $M_1$ and $M_2$ also by using a method
 performed in the proof of Theorem \ref{thm:2}. In this case, the resulting map is different from the map obtained in the proof above.
\end{Rem}

\begin{Ex}
\label{ex:3}
By Theorem \ref{thm:10}, \ref{thm:11} or \ref{thm:12}, the proof of the theorem and Example \ref{ex:2}, a manifold represented
 as a connected sum of $S^n \times S^1$ and another ($n+1$)-dimensional manifold admitting a $C^{\infty}$ trivial
 round fold map into ${\mathbb{R}}^n$ whose image is diffeomorphic to $D^n$ admits a $C^{\infty}$ trivial round fold map
 whose image is diffeomorphic to $D^n$ again where $n \geq 3$ is assumed. In addition, a manifold represented
 as a connected sum of $S^2 \times S^3$ ($S^3 \times S^3$) and another $5$-dimensional (resp. $6$-dimensional) manifold admitting a $C^{\infty}$ trivial
 round fold map into ${\mathbb{R}}^3$ whose image is diffeomorphic to $D^3$ and the fiber of a point in a proper core of which has a
 connected component diffeomorphic to $S^2$ (resp. $S^3$) admits a $C^{\infty}$ trivial round fold map
 whose image is diffeomorphic to $D^3$ and the fiber of a point in a proper core of which has a
 connected component diffeomorphic to $S^2$ (resp. $S^3$) again. 

\ \ \ For example, manifolds represented as connected
 sums of finite copies of $S^n \times S^1$ admit $C^{\infty}$ trivial round fold maps into ${\mathbb{R}}^n$ satisfying all the conditions
 of Theorem \ref{thm:5} (\ref{thm:5.1}) where $n \geq 3$ is assumed. In addition, manifolds represented as
 connected sums of finite copies of $S^3 \times S^2$ and ones represented
 as connected sums of finite copies of $S^3 \times S^3$ admit similar maps into ${\mathbb{R}}^3$.  
\end{Ex}

\end{document}